\documentclass[12pt,a4paper]{article}
\usepackage{amsmath}
\usepackage{amssymb}
\usepackage{amsthm}

\usepackage{verbatim}

\usepackage[utf8]{inputenc}

\usepackage{graphicx}
\usepackage{color}
\usepackage{enumerate}

\usepackage{esint}
\usepackage[all]{xy}
\usepackage{framed}
\usepackage{bbm}
\usepackage{subcaption}

\theoremstyle{plain}
\newtheorem{theorem}{Theorem}[section]
\newtheorem{corollary}[theorem]{Corollary}
\newtheorem{lemma}[theorem]{Lemma}
\newtheorem{claim}{Claim}[section]
\newtheorem{proposition}[theorem]{Proposition}

\newtheorem*{thm_uten_num}{Theorem}

\theoremstyle{definition}
\newtheorem{definition}{Definition}[section]
\newtheorem{remark}{Remark}[section]
\newtheorem{example}{Example}[section]
\numberwithin{equation}{section}

\newcommand{\bfc}{\mathbf{c}}

\newcommand{\bff}{\mathbf{f}}

\newcommand{\bfq}{\mathbf{q}}

\newcommand{\Hv}{\mathcal{H}v}
\newcommand{\cH}{\mathcal{H}}
\newcommand{\cI}{\mathcal{I}}

\newcommand{\dd}{\mathrm{\,d}}

\newcommand{\mR}{\mathbb{R}}
\newcommand{\mE}{\mathbb{E}}
\DeclareMathOperator{\tr}{tr}
\DeclareMathOperator{\diag}{diag}

\DeclareMathOperator{\sgn}{sgn}

\DeclareMathOperator{\spa}{span}

\begin{document}

\author{Karl K. Brustad\\ {\small Norwegian University of Science and Technology}}
\title{Segre's Theorem. An analytic proof of a result in differential geometry }

\maketitle

\begin{abstract}
We present an analytic approach on how to solve the problem $|\nabla u|=f(u)$, $\Delta u = g(u)$, in connected domains $\Omega\subseteq\mathbb{R}^n$.
\end{abstract}

\section{Introduction}

The study of \emph{isoparametric families of hypersurfaces} -- a term used by Levi-Civita in a paper published in 1937 -- traces back to the early 1900s and a problem in geometric optics. The following situation was considered by Laura and Somigliana. Suppose that $u$ satisfies the wave equation in $\mR^3$ and that for any fixed time, the level sets $\{u = c\}$ are parallel. Laura had discovered that the possible level surfaces are strongly restricted, and Somigliana then solved the problem completely by proving that they are either parallel planes, concentric spheres, or coaxial cylinders.
The consequence is that wave fronts are parallel only in very special cases. Considering the simple conditions, the proof of this seemingly simple result, is fascinatingly rich and complex. The problem becomes even more challenging when generalized to $\mR^n$, and this will be the topic of our paper.

The details, and more on the history and the modern development of the subject, can be found in the survey \cite{Thorbergsson2000963} by Thorbergsson. We only mention that the result in space forms with constant negative curvature is more or less the same as for Euclidean space. On the other hand, in $\mathbb{S}^n$ the problem is significantly harder. See 
\cite{MR2999038} and \cite{MR3735906}.

The word ``parallel'' can be replaced with the condition that $|\nabla u|$ is constant on the level sets.
A smooth function \(u\colon M\to\mathbb{R}\) on a Riemannian manifold \(M\) is therefore called \emph{isoparametric} if there are functions \(f\) and \(g\) so that
\begin{equation}
|\nabla u| = f(u)\qquad\text{and}\qquad \Delta u = g(u).
\label{eq:isopar}
\end{equation}
A regular level set of an isoparametric function is called an \emph{isoparametric hypersurface}.

On a general manifold, the gradient $\nabla u$ and the Laplacian $\Delta u$ has to be interpreted in the proper way as the \emph{first}, and \emph{second differential parameter}. However, in this paper we shall only be concerned with the standard interpretations in $\mR^n$.

According to \cite{Thorbergsson2000963}, the classification of isoparametric hypersurfaces in Euclidean space $\mR^n$ was first done by Segre in 1938. He effectively shows that the result of Somigliana still holds true for $n>3$.

\begin{thm_uten_num}[Segre]
	A connected isoparametric hypersurface in \(\mathbb{R}^n\) is, upon scaling and an Euclidean motion, an open part of one of the following hypersurfaces:
	\begin{enumerate}[a)]
		\item a hyperplane \(\mathbb{R}^{n-1}\),
		\item a sphere \(\mathbb{S}^{n-1}\),
		\item a generalized cylinder \(\mathbb{S}^{k-1}\times\mathbb{R}^{n-k}\), \(k = 2,\dots,n-1\).
	\end{enumerate}
\end{thm_uten_num}

The purpose of our paper is to give a detailed, self-contained, and \emph{analytic} proof of this result.
We shall examine the regularity assumptions on $u$, $f$, and $g$ closely, and state the conclusion globally by giving an explicit formula for $u\colon\Omega\to\mR$ in connected domains $\Omega\subseteq\mR^n$. It is perhaps surprising that $u$ does not depend on $g$. This means that $f$ and $g$ are not independent and there will be no solution to the equations \eqref{eq:isopar} unless there is a constitutional relation in the data.
We feel that several of these issues have not been completely setteled in the literature.

In the definitions above, one considers only \emph{regular} level sets. That is, it is assumed that $\nabla u\neq 0$ or, equivalently, $f>0$. We shall not relax on this condition, and it should then be noted that the theorem excludes the smooth isoparametric function $x\mapsto |x|^2$ in any domain $\Omega\subseteq\mR^n$ containing the origin.

As a new feature, it is shown that the theorem also holds in the modern \emph{viscosity} sense when \eqref{eq:isopar} is interpreted as a system  of elliptic PDE's. Some care has to be taken in the setup because there is a subtle difference between the two equations $|\nabla u| - f(u) = 0$ and $f(u) - |\nabla u| = 0$ from the viscosity point of view. It turns out that the variant $\sgn(g(u))\big(f(u) - |\nabla u|\big) = 0$ works when we define $\sgn(t) = 1$ for $t\geq 0$ and $\sgn(t) = -1$ otherwise. For the definitions and the elementary theory of viscosity solutions, we refer to \cite{MR2084272}.

The theorem has two conclusions depending on whether the 1-Laplacian
\[\Delta_1 u := \frac{\Delta u - \Delta_\infty^N u}{|\nabla u|}\]
vanishes or not at some point in $\Omega$. Here, $\Delta_\infty^N u := |\nabla u|^{-2}\nabla u\cH u\nabla u^T$ is the normalized infinity-Laplacian. The 1-Laplacian measures the mean curvature of the level sets of $u$. In the first case we show that $u$ is a function only of $q^Tx$. The immediate consequence is Part a) of Segre's Theorem above. The second case corresponds to the Parts b) and c) as we show that $u$ is a function only of $|R_0(x-x_*)|$. Here we must have $x_*\in\mR^n\setminus\Omega$, so this case is not possible if one starts with $\Omega = \mR^n$. The quantity $R_0$ is a symmetric $n\times n$ projection matrix with rank, or \emph{dimension}, $\tr R_0 = k$ for some $2\leq k\leq n$. The class of symmetric projection matrices will be central in this paper and we shall use the notation
\[Pr(n) := \left\{R\in S(n)\,|\, RR=R\right\}.\]
It consists of the $k$-dimensional subclasses
\[Pr_k(n) := \left\{R\in S(n)\,|\, RR=R,\,\tr R = k\right\},\qquad k=0,\dots,n.\]
Note that $Pr_n(n) = \{I\}$ which means that $u$ is a translated radial function when $k=n$.

\begin{thm_uten_num}
Let $\Omega\subseteq\mR^n$, $n\geq 2$, be open and connected. Suppose there is a function $u\in C^2(\Omega)$ and two one-variable functions $f>0$ and $g$ such that
\begin{equation}
	\begin{cases}
	\lvert \nabla u\rvert = f(u),\\
	\Delta u = g(u),
	\end{cases}\qquad\text{in $\Omega$.}
	\label{eq:mainsyss}
	\end{equation}
Then $f\in C^1(\cI)$ and $g\in C(\cI)$ where $\cI := u(\Omega)\subseteq\mR$ is an open interval.

If $\Delta_1u(x_0) = 0$ at some point $x_0\in\Omega$, then
\[g(t) = f(t)f'(t)\]
and
\[u(x) = U(q^T(x-x_0))\]
for some unit length constant $q\in\mR^n$ and $U$ is the inverse of the function $F\colon\cI\to\mR$ given by
\begin{equation}
F(t) := \int_{u(x_0)}^t\frac{\dd s}{f(s)}.
\label{eq:F}
\end{equation}

If $C_1 := \Delta_1u(x_0) > 0$ at some point $x_0\in\Omega$, then
there is an integer $2\leq k\leq n$ so that
\begin{equation}
g(t) = 
f(t)\left(f'(t) + \frac{k-1}{F_k(t)}\right)
\label{eq:const_rel}
\end{equation}
where
\[
F_k(t) := \frac{k-1}{C_1} + F(t) > 0.
\]
Moreover,
\begin{equation}
u(x) = U_k(|R_0(x-x_*)|)
\label{eq:}
\end{equation}
for some constant projection $R_0\in Pr_k(n)$ and $U_k$ is the inverse of $F_k$. Finally, $x_*$ is a point in $\mR^n$ such that $R_0(x-x_*)\neq 0$ for $x\in\Omega$.
\end{thm_uten_num}

If it happens that your isoparametric function $u$ has a negative 1-Laplacian, apply the theorem to $-u$.

We shall on several occasions exploit the fact that the problem is \emph{geometric}. Meaning that, if $u$ satisfies \eqref{eq:mainsyss}, then $w := G\circ u$ will again be isoparametric for every invertible $C^2$ function $G$. In fact, most of the proof will be carried out by considering the function $v := F\circ u$ with $F$ given by \eqref{eq:F}. This change of dependent variables is not essential for the proof, but many calculations simplifies as the gradient of $v$ has constant length one. By choosing $G$ so that $w := G\circ u$ is harmonic, we are able to show that the Theorem holds in the weaker viscosity interpretation of the equations. The price to pay is that the continuity of $f$ and $g$ has to be assumed.

\begin{proposition}[Segre's theorem in the viscosity sense]
Let $\Omega\subseteq\mR^n$ be open. If $u$ is a viscosity solution to the system
	\begin{equation}
	\begin{cases}
	\sgn(g(u))\big(f(u) - |\nabla u|\big) = 0,\\
	\Delta u - g(u) = 0,
	\end{cases}\qquad\text{in $\Omega$,}
	\end{equation}\label{eq:mainsyss_visc}
where $f>0$ and $g$ are continuous, then $u\in C^2(\Omega)$ and \eqref{eq:mainsyss} holds pointwise.
\end{proposition}

\begin{proof}
For some constants $c_0,c_1$, define the function
\[G(t) := \int_{c_1}^t \exp\left(-\int_{c_0}^\tau \frac{g(s)}{f^{2}(s)}\dd s\right)\dd \tau.\]
Note that $G$ is $C^2$ and that $G'>0$ and $G'' = - G'\frac{g}{f^2}$. Thus $G$ has a $C^2$ inverse, say, $H$, and
$G$ is a solution to the ODE
\begin{equation}
G''(t)f^2(t) + G'(t)g(t) = 0.
\label{eq:G_ode}
\end{equation}

The function $w(x) := G(u(x))$ is continuous in $\Omega$ since the viscosity solution $u$ is continuous by definition. Let $\psi$ be a test function touching $w$ from below at some point $x_0\in\Omega$. The test function $\phi(x) := H(\psi(x))$ touches $u$ from below at $x_0$ since $H$ is increasing. Thus $\Delta\phi(x_0) - g(\phi(x_0)) \leq 0$ and
\[\sgn(g(\phi(x_0)))\big(f(\phi(x_0)) - |\nabla \phi(x_0)|\big) \leq 0.\]
Multiplying the above with the non-negative numbers $|g(\phi(x_0))|$ and $f(\phi(x_0)) + |\nabla \phi(x_0)|$ does not affect the sign, and we conclude that
\[-g(\phi)|\nabla \phi|^2 \leq -g(\phi)f^2(\phi)\]
at $x_0$.
Moreover, $\psi(x) = G(\phi(x))$ and, at $x_0$,
\begin{align*}
\Delta\psi
	&= G''(\phi)|\nabla\phi|^2 + G'(\phi)\Delta\phi\\
	&= G'(\phi)\left(-\frac{g(\phi)}{f^2(\phi)}|\nabla\phi|^2 + \Delta\phi\right)\\
	&\leq G'(\phi)\left(-\frac{g(\phi)}{f^2(\phi)}f^2(\phi) + g(\phi)\right) = 0
\end{align*}
since $G'> 0$. Likewise, if $\psi$ touches $w$ from above, then $\Delta \psi(x_0)\geq 0$ and we have shown that $w$ is a viscosity solution to the equation $\Delta w = 0$. It is known that the viscosity solutions of the Laplace equation are the ordinary harmonic functions. Thus, $w$ is smooth and it follows that $u = H(w)$ is $C^2$ in $\Omega$ and that the equations \eqref{eq:mainsyss} hold pointwise as $\sgn(g)\neq 0$.
\end{proof}

The task is to prove the theorem for $u\in C^2(\Omega)$ and where \eqref{eq:mainsyss} holds pointwise. There are no regularity assumptions on $f$ and $g$, but the simple argument below shows that $f$ is $C^1$. We shall later see that also $g$ is continuous.

After some preliminary results, the proof of the theorem is conducted in two major steps. In Section 3 it is shown that, not only their sum but, each eigenvalue of the Hessian matrix $\cH v$ of $v$ is a function only of $v$ itself. In the language of differential geometry, this corresponds to constant principal curvatures on the level set hypersurface. Our proof is based on the standard ideas as found in \cite{MR3408101} or \cite{MR901710} except that we at some point invoke the \emph{Vandermonde matrix} in order to manage without references to external results concerning the \emph{Bell polynomials} and the \emph{Newton identities} in symmetric polynomials. In order to obtain a global result, as well as to provide the sufficient regularity needed in the next step, we show that $v$ is real-analytic in $\Omega$.

Section 4 constitutes the most challenging part of the proof. Here we show that the number of different eigenvalues of $\cH v$ is at most two. By the complexity of the calculations, one may easily appreciate why some twenty years passed by before Somigliana's result in space was generalized to $\mR^n$. The isoparametric function has to be differentiated four times. In order to keep track and minimize the use of indexes, we introduce a notation for matrix derivatives. This also allows us to exploit the power of linear algebra.
The \emph{principal directions} in the level sets correspond to the eigenvectors of $\cH v$. The existence of differentiable eigenvectors is a nontrivial question. It is addressed, for example, in the book \cite{MR1335452} but only in the case $\Omega \subseteq\mathbb{C}$. We circumvent this problem by doing the calculations directly on the \emph{eigenprojections} $P_j\colon\Omega\to Pr(n)$.
The eigenprojection $P_j(x)$, corresponding to the eigenvalue $\lambda_j(x)$ of $\cH v$ at $x\in\Omega$, is the unique symmetric $n\times n$ projection matrix satisfying
\[\cH v(x)P_j(x) = \lambda_j(x)P_j(x)\]
with dimension equal to the multiplicity of the eigenvalue.

\begin{proposition}\label{prop:fder}
Let $\Omega\subseteq\mR^n$ be open and connected and let $u\in C^2(\Omega)$. If there is a function $f>0$ so that $|\nabla u| = f(u)$ in $\Omega$, then $f\in C^1(\cI)$ where $\cI := u(\Omega)$ is an open interval.
\end{proposition}

\begin{proof}
Firstly, $\cI$ is a (possibly unbounded) open interval because $\Omega$ is connected and $u\in C^2(\Omega)$ is without critical points.

Let $t_0\in\cI$ and choose $x_0\in\Omega$ so that $t_0 = u(x_0)$. Since $\nabla u$ is locally Lipschitz, the problem
\[\bfc'(\tau) = \nabla u^T(\bfc(\tau)),\qquad \bfc(0) = x_0,\]
is well defined for small $|\tau|$.
Define $h(\tau) := u(\bfc(\tau))$. We see that $h$ is $C^2$ with
\[h' = \nabla u(\bfc)\bfc' = |\nabla u(\bfc)|^2 = f^2(u(\bfc)) = f^2(h).\]
In particular, $h'>0$ and $h$ has a $C^2$ inverse.
Thus,
\[f(t) = f\big(h(h^{-1}(t))\big) = \sqrt{h'(h^{-1}(t))} = \frac{1}{\sqrt{(h^{-1})'(t)}}\]
which is continuously differentiable in a neighbourhood of $t_0$. It follows that $f\in C^1(\cI)$ since $t_0\in\cI$ was arbitrary.
\end{proof}

\section{Preliminaries}\label{sec:prelim}


The content of this Section is partially copied from \cite{bru19}. Here, $\Omega$ always denotes an open subset of $\mR^n$.

\subsection{Matrix derivatives}

If $\bff\colon \Omega\to\mR^m$ is a differentiable function, its \emph{Jacobian matrix} is the mapping $\nabla\bff\colon\Omega\to\mR^{m\times n}$ satisfying
\[\bff(x+y) = \bff(x) + \nabla\bff(x)y + o(\lvert y\rvert)\]
as $\mR^n\ni y\to 0$. In particular, gradients are row vectors.

\begin{definition}
Suppose that $H\colon \Omega\to\mathbb{R}^{m\times k}$ is differentiable.
The \emph{Jacobian derivative} $\nabla H\colon\mR^k\times\Omega\to\mR^{m\times n}$ of $H$ is defined by
\begin{equation}\label{def1}
\nabla_q H(x) := \nabla[Hq](x).
\end{equation}
That is, the Jacobian matrix of the vector valued function $x\mapsto H(x)q$.
\end{definition}

It is possible to define the Jacobian in terms of combinations of partial derivatives, but we shall reserve the notation $\nabla$ and $\nabla_q$ for functions that are assumed to be differentiable.


If $\bfq\colon\Omega\to\mathbb{R}^k$ is a function, we write
\[\nabla_{\bfq(x)} H(x) := \nabla_q H(x)\Big\vert_{q = \bfq(x)}.\]
Thus if $\bfq$ is differentiable, the product rule yields
\[\nabla[H\bfq](x) = H(x)\nabla\bfq(x) + \nabla_{\bfq(x)} H(x).\]
Moreover, for vectors $p\in\mR^m$ and $q\in\mR^k$ we have
\begin{equation}\label{sym2}
p^T\nabla_q H(x) = q^T\nabla_p H^T(x).
\end{equation}
Note that the dimensions match and that \eqref{sym2} is an equality in $\mR^{1\times n}$. Indeed, since $H^T(x)$ is a $k\times m$ matrix, the Jacobian $\nabla_p H^T = \nabla[H^Tp]$ is of dimension $k\times n$.

The Jacobian derivative of the transposed of a Jacobian matrix is symmetric. That is,
\[\nabla_p(\nabla \bff)^T\colon \Omega\to S(n)\]
for $\bff\in C^2(\Omega,\mR^m)$ and $p\in\mR^m$. This is because it is the Hessian of the function $x\mapsto p^T\bff(x)$.
In particular,
\begin{equation}
\nabla_q\cH u\colon\Omega\to S(n)\qquad\text{and}\qquad \nabla_p(\nabla_q H)^T\colon\Omega\to S(n)
\label{sym1}
\end{equation}
for $u\in C^3(\Omega)$ and $H\in C^2(\Omega,\mR^{m\times k})$.

\subsection{Symmetric matrices}

The \emph{spectral theorem} states that every symmetric $n\times n$ matrix can be diagonalized.
For any $X\in S(n)$ there exists an orthogonal $n\times n$ matrix $U$ such that $U^TXU = \diag(\lambda_1,\dots,\lambda_n)$ where $\lambda_1\leq\cdots\leq\lambda_n$ are the eigenvalues of $X$.
Moreover, the \emph{eigenspaces} $\mE_j := \{\xi\in\mR^n\,\vert\, X\xi = \lambda_j\xi\}$ are $d_j$-dimensional subspaces of $\mR^n$ where $d_j$ is the multiplicity of $\lambda_j$.
The spaces $\mE_j$ and $\mE_k$ are orthogonal whenever $\lambda_j\neq \lambda_k$. Obviously, $\mE_j=\mE_k$ if $\lambda_j= \lambda_k$.
By writing $U = (\xi_1,\dots,\xi_n)$, we get that
\begin{equation}\label{eq:Xdiag}
X = U\diag(\lambda_1,\dots,\lambda_n)U^T = \sum_{i=1}^n \lambda_i\xi_i\xi_i^T
\end{equation}
and that $\mE_j = \spa\{\xi_i\,\vert\, \lambda_i = \lambda_j\}$.

The class of symmetric $n\times n$ projection matrices is denoted by
\[Pr(n) := \{R\in S(n)\,\vert\, RR = R\}.\]
Since their eigenvalues are either 0 or 1, these matrices are on the form
\begin{equation}\label{eq:projsplit}
R = \sum_{i=1}^d \xi_i\xi_i^T = QQ^T,\qquad Q := (\xi_1,\dots,\xi_d)\in\mR^{n\times d},
\end{equation}
for some $d = 0,1,\dots, n$ (with the convention that empty sums are zero) and where $Q^TQ = I_d$. The set $\{\xi_1,\dots,\xi_d\}$ is an orthonormal basis for the $d$-dimensional subspace
\[R(\mR^n) := \{R\xi\,\vert\, \xi\in\mR^n\}\subseteq \mR^n.\]
Conversely, given a subspace $\mE$ of $\mR^n$, there is a unique symmetric projection $R$ such that $\mE = R(\mR^n)$.
Indeed, if $P(\mR^n) = \mE = R(\mR^n)$, then $P\xi,R\xi\in \mE$ for every $\xi\in\mR^n$. Thus $RP\xi = P\xi$ and $PR\xi = R\xi$ and $P = P^T = (RP)^T = PR = R$.
Note therefore that the factorization \eqref{eq:projsplit} is not unique as $R = \sum_{i=1}^d \eta_i\eta_i^T$ for every orthonormal basis $\{\eta_1,\dots,\eta_d\}$ of $R(\mR^n)$.

In the case of the symmetric matrix $X$ it follows that
\[R_j = \sum_{\substack{ i=1\\ \lambda_i = \lambda_j} }^n\xi_i\xi_i^T\]
is the unique eigenprojection corresponding to the $j$'th eigenvalue of $X$, regardless of the choice $U = (\xi_1,\dots,\xi_n)$ of eigenvectors.

If we let $\alpha\colon\{1,\dots,s\}\to \{1,\dots,n\}$ be a re-indexing that picks out all of the $s := \vert\{\lambda_1,\dots,\lambda_n\}\vert$ \emph{distinct} eigenvalues of $X$, we may collect the terms in \eqref{eq:Xdiag} with equal coefficients and write
\begin{equation}\label{eq:X_unrep}
X = \sum_{i=1}^s \kappa_i P_i,\qquad\text{where }\, \kappa_i := \lambda_{\alpha(i)},\, P_i := R_{\alpha(i)}.
\end{equation}

Now,
\[P_iP_j = \delta_{ij}P_i\qquad \text{and}\qquad \sum_{l=1}^sP_i = \sum_{i=1}^n\xi_i\xi_i^T = I,\]
and \eqref{eq:X_unrep} is the unique representation of $X$ in terms of a \emph{complete set of eigenprojections} $\{P_i\}_{i=1}^s$ and the \emph{unrepeated} eigenvalues $\{\kappa_i\}_{i=1}^s$.
Note that if $X = H(x)$ is a matrix valued function, then the re-indexing $\alpha$ also depends on $x$ and some care is needed when using the formula \eqref{eq:X_unrep}.
However, if the number of distinct eigenvalues is known to be constant, the lemma below shows that a continuous $H$ is on the form \eqref{eq:X_unrep} globally on connected domains. This is not completely obvious since we have to rule out the possibility that two pairs of eigenvalues merge and split simultaneously.



In Section 4 we are going to differentiate the eigenprojections of the Hessian matrix $\cH v$. A regularity result for $P_i$ is therefore also needed.

\begin{lemma}\label{lem:unirep}
Suppose that $H\colon E\to S(n)$ is continuous on a connected subset $E\subseteq\mR^n$ and has a constant number $s$ of distinct eigenvalues. Then the multiplicity of each eigenvalue is also constant and $H$ has the unique representation 
\begin{equation}
H(x) = \sum_{i=1}^s\kappa_i(x)P_i(x)
\label{eq:uni}
\end{equation}
on $E$ where $\kappa_1(x)<\cdots<\kappa_s(x)$ are the unrepeated eigenvalues and where $\{P_i(x)\}_{i=1}^s$ is a complete set of eigenprojections.

Moreover, if \eqref{eq:uni} holds in and open subset $\Omega\subseteq\mR^n$, and $H$ and each $\kappa_i\colon\Omega\to\mR$ is $C^k$ in $\Omega$, $k\geq 0$, then every $P_i\colon\Omega\to Pr(n)$ is also $C^k$ in $\Omega$.
\end{lemma}

\begin{proof}
For $x\in E$, let
$\lambda_1(x)\leq\cdots\leq\lambda_n(x)$ be the repeated eigenvalues of $H(x)$ and let $d_j\colon E\to\{1,\dots,n\}$ be the multiplicity of $\lambda_j$. We see that $d_j$ increases only if two different eigenvalues become equal. Thus $d_j$ is upper semicontinuous (u.s.c.) on $E$ since the eigenvalues are continuous (It is a standard result that eigenvalues depends continuously on the matrix. Thus $x\mapsto\lambda_j(x)$ is continuous since $H$ is continuous.)

For each $x\in E$ choose a re-indexing $\alpha\colon\{1,\dots,s\}\to\{1,\dots,n\}$ so that $l\mapsto \lambda_{\alpha(l)}$ is a bijection.
Since, for every $l$,
\[\sum_{\substack{ i=1\\ \lambda_i = \lambda_{\alpha(l)}} }^n\frac{1}{d_i} = \frac{1}{d_{\alpha(l)}}\sum_{\substack{ i=1\\ \lambda_i = \lambda_{\alpha(l)}} }^n 1 = \frac{1}{d_{\alpha(l)}}d_{\alpha(l)} = 1,\]
we get that
\[s = \sum_{l=1}^s 1 = \sum_{l=1}^s\sum_{\substack{ i=1\\ \lambda_i = \lambda_{\alpha(l)}} }^n\frac{1}{d_i(x)} = \sum_{i=1}^n\frac{1}{d_i(x)}.\]
Now,
each \(d_i\) is u.s.c. which means that \(\frac{1}{d_i}\) is l.s.c. and \(-\frac{1}{d_i}\) is again u.s.c. So as $s$ is constant,
\[\frac{1}{d_j(x)} = s - \sum_{ \substack{i=1\\ i\neq j} }^m\frac{1}{d_i(x)}\]
is u.s.c. Thus \(d_j\) is also l.s.c. and therefore continuous. The multiplicity of the eigenvalues are therefore constant on the connected set $E$ and we may define an increasing re-indexing $\alpha\colon\{1,\dots,s\}\to\{1,\dots,n\}$ that does not depend on $x$. 

Next, since the factors in the quantity
\[\prod_{\substack{l=1\\ l\neq i}}^{s}\frac{H - \kappa_lI}{\kappa_i - \kappa_l}\]
commute and $(H - \kappa_jI)P_j = 0$, the product will be zero when multiplied with $P_j$ for all $j\neq i$. On the other hand, multiplying with $P_i$ yields
\[\prod_{\substack{l=1\\ l\neq i}}^{s}\frac{H - \kappa_lI}{\kappa_i - \kappa_l}P_i = \prod_{\substack{l=1\\ l\neq i}}^{s}\frac{\kappa_i - \kappa_l}{\kappa_i - \kappa_l}P_i = P_i,\]
and thus
\[\prod_{\substack{l=1\\ l\neq i}}^{s}\frac{H - \kappa_lI}{\kappa_i - \kappa_l} = \prod_{\substack{l=1\\ l\neq i}}^{s}\frac{H - \kappa_lI}{\kappa_i - \kappa_l}I = \prod_{\substack{l=1\\ l\neq i}}^{s}\frac{H - \kappa_lI}{\kappa_i - \kappa_l}\sum_{j=1}^sP_j = P_i\]
which shows that $P_i$ has the same $C^k$ regularity as $H$ and the eigenvalues. With the convension that empty products is the identity, the formula also holds in the case $s=1$.
\end{proof}

We remark that the last part of the Lemma is true without any regularity assumption on the eigenvalues, but the proof is then more involved. See \cite{bru19} (Prop. 5.2.)

The formula $P_i = \prod_{l\neq i}^{s}(H - \kappa_lI)(\kappa_i - \kappa_l)^{-1}$ appears in \cite{MR1091716}, where the unrepeated eigenprojections are called the \emph{Frobenius covariants}.

\section{The eigenvalues are constant on level sets}\label{sec:eigconst}

The main result of this Section is

\begin{proposition}\label{prop:eigenvalues are constant on level sets}
Let $\Omega\subseteq\mR^n$ be open and connected. Assume that $v\in C^2(\Omega)$ with $\vert\nabla v\vert \equiv 1$ and $\Delta v(x) = \tilde{g}(v(x))$ for some function $\tilde{g}$. Then $v$ is real-analytic in $\Omega$. Moreover, if $v(x_0) = 0$ at some point $x_0\in\Omega$, then there is a neighbourhood of $x_0$ in which the Hessian matrix of $v$ is on the form
\begin{equation}
\Hv(x) = \sum_{i=0}^m\frac{c_i}{1 + c_iv(x)}P_i(x),\qquad \{P_i\}_{i=0}^m\, \text{ c.s.e.},
\label{eq:prop21}
\end{equation}
for some $m\geq 0$ distinct nonzero constants $c_1, \dots,c_m$, and where $c_0:=0$.
\end{proposition}

For $x,y$ in $\mR^n$, we let $[x, y] := \{ sx + (1-s)y\,\vert\, 0\leq s\leq 1\}$ denote the line segment joining the two points.

\begin{lemma}\label{lem:disteig}
Let $v\in C^2(\Omega)$ with $\vert\nabla v\vert \equiv 1$.
The following holds for all $x\in\Omega$ and all $t\in\mathbb{R}$ such that $[x, x+t\nabla v(x)]\subseteq\Omega$.
\begin{enumerate}[(1)]
\item $v(x+t\nabla v^T(x)) = v(x) + t$.
\item $\nabla v(x+t\nabla v^T(x)) = \nabla v(x)$.
\item If $\cH v(x) = \sum_{i=0}^m\kappa_i(x)P_i(x)$, then
\[\Hv\Big(x+t\nabla v^T(x)\Big) = \sum_{i=0}^{m}\frac{\kappa_i(x)}{1+t\kappa_i(x)}P_i(x).\]
\end{enumerate}
\end{lemma}

In Part (3), the number $m+1$ of distinct eigenvalues of $\cH v$ generally depends on $x$. Our claim is therefore that $m$ and the eigenprojections $P_i$ are constant in the gradient direction.

We remark that similar results can be found in \cite{MR0473443} (Appendix), and in  \cite{MR3155251} (Thm. 1.18.)

\begin{proof}
Let  $x\in\Omega$. Since $\nabla v$ is locally Lipschitz, the integral line $\bfc$ satisfying
\[\bfc'(t) = \nabla v^T(\bfc(t)),\qquad \bfc(0) = x,\]
is well defined as long as $\bfc(t)\in\Omega$. Since $\bfc'' = \Hv(\bfc)\nabla v^T(\bfc) = 0$ we get that $\bfc'(t) = \nabla v^T(\bfc(t)) = \nabla v^T(x)$ is constant and thus $\bfc(t) = x + t\nabla v^T(x)$. That is,
\[\nabla v^T(x) = \bfc'(t) = \nabla v^T(\bfc(t)) = \nabla v^T(x + t\nabla v^T(x)).\]
The first claim now follows as
\[\frac{\dd}{\dd t}v(x + t\nabla v^T(x)) = \nabla v(x + t\nabla v^T(x))\nabla v^T(x) = \nabla v(x)\nabla v^T(x) = 1.\]
	
(3) Write $z := x+t\nabla v^T(x)$. By differentiating the identity in (2) we get
\[\Hv(x) = \Hv\left(z\right)\big(I + t\Hv(x)\big),\]
and multiplying from the right with $P_i(x)$ gives
\[\kappa_i(x)P_i(x) = (1+t\kappa_i(x))\Hv(z)P_i(x).\]
Dividing by $1+t\kappa_i$ and summing over $i$ then yields the result:
\[\cH v(z) = \cH v(z)\sum_{i=0}^m P_i(x) = \sum_{i=0}^m \frac{\kappa_i(x)}{1+t\kappa_i(x)}P_i(x).\]
\end{proof}

\begin{proof}[Proof of Proposition \ref{prop:eigenvalues are constant on level sets}]
Let $x_0\in\Omega$ and suppose that $v(x_0) = 0$. Define $E = E(x_0)\subseteq\Omega$ to be the connected component of the zero level set $\{x\in\Omega\,\vert\, v(x) = 0\}$ containing $x_0$. Expand this level set in the gradient direction and write
\[\Omega' = \Omega'(x_0) := \Big\{ x + t\nabla v(x)\,\vert\, x\in E(x_0),\, t\in\mR,\, [x, x+t\nabla v(x)]\subseteq\Omega\Big\}.\]
For $x\in E$, denote by $ s(x)$ the number of \emph{distinct} eigenvalues of $\Hv(x)$.
Since $\Hv$ always has a zero eigenvalue,
\[m := \max_{x\in E} s(x)-1\]
will be the largest number of distinct and nonzero eigenvalues on the level set $E$. Let $E_m$ be a nonempty connected component of $\{x\in E\,\vert\,  s(x)-1 = m\}$.
In order to go from $ E_m$ to $E\setminus  E_m$, the number $ s(x)$ has to decrease which again means that two distinct eigenvalues becomes equal. Thus $E_m$ is relatively open since the eigenvalues are continuous. We want to show that $E_m = E$. This is immediate if $m=0$, and by part (3) of Lemma \ref{lem:disteig} we get that $\cH v = 0$ in $\Omega'$. Assume therefore in the following that $m\geq 1$.
	
By Lemma \ref{lem:unirep} the Hessian matrix of $v$ has the unique representation
\[\Hv(x) = \sum_{i=1}^m\kappa_i(x)P_i(x),\qquad x\in E_m,\]
where $\kappa_1(x) < \cdots < \kappa_m(x)$ denote the distinct nonzero eigenvalues.
	
For $x\in E_m$ and $t\in\mR$ such that $x + t\nabla v^T(x)\in\Omega'$ let $X$ be the matrix valued function
\[X(x,t) := \Hv(x+t\nabla v^T(x)).\]
By part (3) of Lemma \ref{lem:disteig}, we have
\[X(x,t) = \sum_{i=1}^m\frac{\kappa_i(x)}{1+t\kappa_i(x)}P_i(x)\]
which is smooth in $t$ for small $|t|$.
Note also that
\[\tilde{g}(t) = \tilde{g}(v(x) + t) = \tilde{g}\left( v(x+t\nabla v^T(x)) \right) = \tr X(x,t)\]
by part (1) of Lemma \ref{lem:disteig} and since $v(x) = 0$ for $x\in E_m\subseteq E$. Thus $\tilde{g}$ is also smooth close to $t=0$.
Next,
\begin{align*}
\frac{\partial}{\partial t}X(x,t) &= -\sum_{i=1}^m\left(\frac{\kappa_i(x)}{1+t\kappa_i(x)}\right)^2P_i(x)\\
&= -X^2(x,t)
\end{align*}
and by induction,
\[\frac{\partial^k}{\partial t^k}X = (-1)^k k!X^{k+1}.\]
Setting $d_i := \tr P_i(x)$ -- the number of eigenvalues equal to $\kappa_i(x)$ -- which also is constant on $E_m$ by Lemma \ref{lem:unirep}, we get that
	\begin{align*}
	g(t) &= \tr X(x,t) = \sum_{i=1}^m d_i\frac{\kappa_i(x)}{1+t\kappa_i(x)},\\
	g'(t) &= -\tr X^2(x,t) = -\sum_{i=1}^m d_i\left(\frac{\kappa_i(x)}{1+t\kappa_i(x)}\right)^2,\\
	&\:\:\vdots\\
	g^{(k)}(t) &= (-1)^k k!\tr X^{k+1}(x,t) = (-1)^k k!\sum_{i=1}^m d_i\left(\frac{\kappa_i(x)}{1+t\kappa_i(x)}\right)^{k+1}.
	\end{align*}
	In particular,
	\begin{align*}
	C_1 &:= g(0) = \sum_{i=1}^m d_i\kappa_i(x),\\
	C_2 &:= -g'(0) = \sum_{i=1}^m d_i\kappa_i^2(x),\\
	&\:\:\vdots\\
	C_m &:= \frac{(-1)^{m-1}}{(m-1)!}g^{(m-1)}(0) = \sum_{i=1}^m d_i\kappa_i^m(x).
	\end{align*}
Define the vector field $\bff\colon \mathbb{R}^m\to\mathbb{R}^m$, $\bff(y) = (f^1(y),\dots,f^m(y))^T$, as
	\[f^k(y) = f^k(y_1,\dots,y_m) := \sum_{i=1}^m  d_iy_i^k,\qquad k = 1,\dots, m.\]
	Clearly, $\bff$ is smooth and the partial derivatives of the components are
\[\frac{\partial}{\partial y_j} f^k(y) = k d_jy_j^{k-1}.\]
The Jacobian matrix of $\bff$ can therefore be written and factorized as
\begin{align*}
	\nabla\bff(y) &= 
	\begin{pmatrix}
	 d_1 &  d_2 & \cdots &  d_m\\
	2 d_1y_1 & 2 d_2y_2 &\cdots & 2 d_m y_m\\
	3 d_1y_1^2 & 3 d_2y_2^2 &\cdots & 3 d_my_m^2\\
	\vdots & \vdots & &\vdots\\
	m d_1y_1^{m-1} & m d_2y_2^{m-1} &\cdots & m d_my_m^{m-1}
	\end{pmatrix}\\ &= \diag(1,\dots,m)
	\begin{pmatrix}
	1 & 1 & \cdots & 1\\
	y_1 & y_2 &\cdots & y_m\\
	y_1^2 & y_2^2 &\cdots & y_m^2\\
	\vdots & \vdots & &\vdots\\
	y_1^{m-1} & y_2^{m-1} &\cdots & y_m^{m-1}
	\end{pmatrix}\diag( d_1,\dots, d_m).
	\end{align*}
The middle factor is the the \emph{Vandermonde matrix}
and is known to have determinant $\prod_{1\leq i<j\leq m}(y_j-y_i)$. Thus
	\[\det \nabla\bff(y) = m!\, d_1\cdots d_m \prod_{1\leq i<j\leq m}(y_j-y_i)\]
	which is nonzero if and only if the $y_i$'s are all distinct.
This is the case when evaluated at $y_i = \kappa_i(x)$. The inverse $\bff^{-1}$ then exists in a neighbourhood of the image, and since
	\[\bff(\kappa_1(x),\dots,\kappa_m(x)) = (C_1,\dots, C_m)^T,\]
	it follows that
	\[\begin{pmatrix}
	\kappa_1(x)\\
	\vdots\\
	\kappa_m(x)
	\end{pmatrix} = \bff^{-1}(C_1,\dots, C_m) =:
	\begin{pmatrix}
	c_1\\
	\vdots\\
	c_m
	\end{pmatrix}\]
	for all $x\in E_m$.
	
	We have shown that, in the relatively open set $ E_m\subseteq E$ -- where the number of distinct eigenvalues is maximal -- each eigenvalue is constant. Being continuous, they are naturally constant in the closure in which they are still distinct. Thus $ E_m$ is also relatively closed in $E$, and since the level set was assumed to be connected, we must have $ E_m = E$.
	
Next, we note that for $z\in\Omega'$ we can choose $x\in E$ and $t\in\mR$ so that $z = x + t\nabla v^T(x)$. Then $v(z) = t$, $P_i(z) = P_i(x)$, and
\begin{align*}
\Hv(z) &= \Hv\left(x + t\nabla v^T(x)\right)\\
	&= \sum_{i=1}^{m}\frac{\kappa_i(x)}{1+t\kappa_i(x)}P_i(x)\\
	&= \sum_{i=1}^{m}\frac{c_i}{1+c_iv(z)}P_i(z).
\end{align*}
That is,
\[\cH v = \sum_{i=1}^{m}\frac{c_i}{1+c_iv}P_i\]
in the neighbourhood $\Omega'(x_0)$ of $x_0$ and
\[\tilde{g}(t) = \sum_{i=1}^{m}d_i\frac{c_i}{1+c_it}, \qquad d_i := \tr P_i = const.,\]
for small $|t|$.

Now we turn to the real-analyticity. Define $\tilde{G}$ to be the analytic function
\[\tilde{G}(t) := \int_0^t\prod_{i=1}^m(1+c_i\tau)^{-d_i}\dd \tau,\qquad \text{$|t|$ small}.\]
We have that $\tilde{G}' > 0$ and that $\ln\tilde{G}'(t) = - \sum_{i=1}^m  d_i \ln(1+c_it)$. Thus,
\[\frac{\tilde{G}''(t)}{\tilde{G}'(t)} = \frac{\dd}{\dd t}\ln\tilde{G}'(t) = - \sum_{i=1}^m  d_i \frac{\dd}{\dd t}\ln(1+c_it) = - \sum_{i=1}^m  d_i\frac{c_i}{1+c_it} = -\tilde{g}(t).\]
If we now let $w := \tilde{G}\circ v$, then $\cH w = \tilde{G}''(v)\nabla v^T\nabla v + \tilde{G}'(v)\cH v$ and it follows that
\[\Delta w = \tilde{G}''(v) + \tilde{G}'(v)\tilde{g}(v) = 0.\]
Since $\tilde{G}'>0$, the function $\tilde{G}$ has an analytic inverse $\tilde{G}^{-1}$ and $v = \tilde{G}^{-1}\circ w$ is therefore real-analytic in $\Omega'$.

We have shown that if $v\in C^2(\Omega)$ with $|\nabla v|=1$ and where $\Delta v$ is constant on the level sets, then $v$ is real-analytic in a neighbourhood of every point on its null-level set. Clearly, for other points $x_1\in\Omega$ with $c := v(x_1)$, the function $v_c := v-c$ also satisfies the conditions and is therefore real-analytic near $x_1$.
Thus $v$ is real-analytic around every point in $\Omega$.
\end{proof}

\section{The Hessian matrix has atmost one nonzero eigenvalue}

\begin{proposition}\label{prop:mis1}
Let $\Omega\subseteq\mR^n$ be open. Assume that $v\in C^4(\Omega)$ with $\vert\nabla v\vert \equiv 1$ and with Hessian matrix
\[\Hv(x) = \sum_{i=0}^m\frac{c_i}{1 + c_iv(x)}P_i(x)\]
for $m\geq 1$ distinct nonzero constants $c_1< \cdots<c_m$ and where $c_0=0$. Then $m=1$.
\end{proposition}

Write
\[H(x) := \Hv(x) = \sum_{i=0}^m \kappa_i(x)P_i(x)\]
where $\kappa_i := c_i(1 + c_iv)^{-1}$. Note that $\nabla\kappa_i = -\kappa_i^2\nabla v$ and that $\kappa_0\equiv 0$.

By Lemma \ref{lem:unirep} the eigenprojections $P_i\colon\Omega\to Pr(n)$ are $C^2$.
Since $\nabla v H= 0$, we have $\nabla v P_0 = \nabla v$ and that $\nabla v P_i = 0$ for $i\geq 1$. For $0\leq i\leq m$, define the \emph{pseudo inverses}
\[H_i^\dagger(x) := \sum_{\substack{k=0\\ k\neq i}}^m\frac{P_k(x)}{\kappa_i(x)-\kappa_k(x)}.\]
They commute with each other and $H$, and satisfies
\[(\kappa_i I - H)H_i^\dagger = I-P_i\]
as can easily be checked.

\begin{lemma}\label{lem:M}
Let $q\in\mR^n$. Then
\begin{align}
\nabla_q P_iP_i &= -\kappa_i\left(\nabla v q\cdot I + \nabla v^Tq^T\right)P_i,\label{eq:accel_lem}\\
\nabla_{P_i q} P_i P_j &= (\kappa_i-\kappa_j)H^\dagger_i(\nabla_q P_i)^TP_j\label{eq:accel_lemII},
\end{align}
for all $i,j=0,\dots,m$, $i\neq j$.
\end{lemma}

\begin{proof}[Proof of Lemma]
Let $i=0,\dots,m$. Differentiating $HP_i = \kappa_i P_i$ gives $H\nabla_q P_i + \nabla_{P_iq} H = \kappa_i\nabla_q P_i - \kappa_i^2P_iq\nabla v$. Thus,
\begin{equation}\label{eq:Hid}
\begin{aligned}
\nabla_{P_iq} H &= (\kappa_i I - H)\nabla_q P_i - \kappa_i^2P_iq\nabla v\\
&= (\nabla_q P_i)^T(\kappa_i I - H) - \kappa_i^2\nabla v^T q^T P_i
\end{aligned}
\end{equation}
since $\nabla H$ is symmetric by \eqref{sym1}. Multiply from the right by $P_i$ and from the left by $H_i^\dagger$ to get
\[(I-P_i)\nabla_q P_iP_i = -\kappa_i^2H_i^\dagger\nabla v^T q^T P_i = -\kappa_i\nabla v^T q^T P_i.\]
Next, $P_i = P_iP_i$ so $\nabla P_i = P_i\nabla P_i + \nabla_{P_i} P_i$, or $\nabla_{P_i} P_i = (I-P_i)\nabla P_i$. That is, $\nabla_{P_ir} P_iP_i = -\kappa_i\nabla v^T r^T P_i$ for all $r\in\mR^n$.  Multiply from the left by $q^T$ and use the symmetry rule \eqref{sym2}:
\[-\kappa_i(\nabla v q) r^T P_i = q^T\nabla_{P_ir} P_iP_i = r^TP_i\nabla_q P_iP_i.\]
The formula \eqref{eq:accel_lem} now follows by writing $\nabla P_i P_i = (P_i + I-P_i)\nabla P_i P_i$.

The identity \eqref{eq:accel_lemII} is obtained similarly by multiplying \eqref{eq:Hid} from the right by $P_j$ and from the left by $H_i^\dagger$.
\end{proof}

Since $P_jP_i = 0$ for $i\neq j$, we have that $P_j\nabla P_i + \nabla_{P_i}P_j = 0$. Using \eqref{sym2} yields
\begin{equation}
e^TP_j\nabla_q P_i = - q^T P_i\nabla_e P_j,\qquad\text{for all $i\neq j$ and all $q,e\in\mR^n$.}
\label{eq:qmin}
\end{equation}

\begin{proof}[Proof of the Proposition]
Let $i\in\{1,\dots,m\}$ and $j\in\{0,\dots,m\}$ with $i\neq j$.
Since $P_j\nabla P_i = -\nabla_{P_i}P_j$ and $i\neq 0$ we get from \eqref{eq:accel_lem} in the Lemma, when multiplying on the right with $P_j$, that
\begin{equation}
P_j\nabla_q P_iP_j = 0
\label{eq:qw}
\end{equation}
for all $q\in\mR^n$. Similarly, multiplying \eqref{eq:accel_lem} from the left with $P_j$ yields zero unless $j=0$. More precisely, and after transposing, we have that
\[P_i(\nabla_q P_i)^T P_j = -\delta_{j,0}\,\kappa_iP_iq\nabla v.\]
The derivative of the right-hand side matrix is
\[\nabla_e[-\delta_{j,0}\,\kappa_i P_iq\nabla v] = -\delta_{j,0}\,\kappa_i\Big\{P_iqe^TH + \nabla ve\cdot\nabla_q P_i\Big\} + \delta_{j,0}\,\kappa_i^2\nabla ve \cdot P_iq\nabla v.\]
Choose $q$ such that $P_iq = q$ and $|q| = 1$.
Then the above reduces to
\[P_i\nabla_e[-\delta_{j,0}\,\kappa_i P_iq\nabla v]P_j = \delta_{j,0}\,\kappa_i^2\nabla ve \cdot P_iq\nabla v\]
after multiplying from the left and right by $P_i$ and $P_j$, respectively. The term $P_i\nabla_q P_iP_j$ disappears by \eqref{eq:accel_lemII}.

For general $i,j,k$ the product rule yields
\[\nabla_e[P_k(\nabla_q P_i)^T P_j] = P_k\Big\{(\nabla_q P_i)^T \nabla_eP_j + \nabla_{P_je}(\nabla_q P_i)^T\Big\} + \nabla_s P_k\]
where $s^T$ is given by \eqref{eq:qmin}.
Therefore, by subtracting the transposed of the derivative of \eqref{eq:qw}, we are left with the identity
\begin{equation}
\begin{aligned}
\delta_{j,0}\,\kappa_i^2\nabla ve \cdot P_iq\nabla v
	&= P_i\Big\{ \nabla_e[P_i(\nabla_q P_i)^TP_j] - \left(\nabla_e[P_j(\nabla_q P_i)^TP_j]\right)^T \Big\}P_j\\
	&= P_i\Big\{ (\nabla_q P_i)^T\nabla_eP_j + \nabla_{P_je}(\nabla_q P_i)^T + \nabla_s P_i\\
	&\quad{} - (\nabla_e P_j)^T\nabla_qP_i -\nabla_{P_je}(\nabla_q P_i)^T - \left(\nabla_s P_j\right)^T \Big\}P_j\\
	&= P_i\Big\{ (\nabla_q P_i)^T\nabla_eP_j - (\nabla_e P_j)^T\nabla_qP_i + \nabla_s P_i - \left(\nabla_s P_j\right)^T \Big\}P_j
\end{aligned}
\label{eq:er}
\end{equation}
since $\nabla(\nabla P)^T$ is symmetric by \eqref{sym1}. Multiplying the last term with $q^T$ and and arbitrary $p$ gives
\[q^TP_i(\nabla_s P_j)^T P_jp = q^TP_i(\nabla_{P_jp} P_j)^T s = (\kappa_j - \kappa_i)q^TP_i\nabla_{P_jp} P_j H_j^\dagger s = (\kappa_i - \kappa_j)p^TP_j\nabla_{q} P_i H_j^\dagger s\]
by \eqref{eq:accel_lemII} and thus
\[q^TP_i(\nabla_s P_j)^T P_j = (\kappa_i-\kappa_j)s^TH_j^\dagger (\nabla_q P_i)^TP_j.\]
Similarly, considering the second to last term,
\[q^TP_i\nabla_s P_i P_j = s^T\nabla_q P_i P_j = (\kappa_i-\kappa_j)s^TH_i^\dagger (\nabla_q P_i)^T P_j.\]
Next, 
\begin{align*}
H_i^\dagger - H_j^\dagger
	&= \sum_{\substack{k=0\\ k\neq i}}^m\frac{P_k}{\kappa_i - \kappa_k} - \sum_{\substack{k=0\\ k\neq j}}^m\frac{P_k}{\kappa_j - \kappa_k}\\
	&= \frac{P_j}{\kappa_i - \kappa_j} - \frac{P_i}{\kappa_j - \kappa_i} + \sum_{\substack{k=0\\ i\neq k\neq j}}^m\frac{P_k}{\kappa_i - \kappa_k} - \frac{P_k}{\kappa_j - \kappa_k}\\
	&= \frac{P_i + P_j}{\kappa_i - \kappa_j} - (\kappa_i - \kappa_j)\sum_{\substack{k=0\\ i\neq k\neq j}}^m\frac{P_k}{(\kappa_i - \kappa_k)(\kappa_j - \kappa_k)}\\
	&= \frac{P_i + P_j}{\kappa_i - \kappa_j} - (\kappa_i - \kappa_j)H_i^\dagger H_j^\dagger,
\end{align*}
and it follows that
\begin{align*}
q^TP_i\Big\{ \nabla_s P_i - \left(\nabla_s P_j\right)^T \Big\}P_j
	&= (\kappa_i-\kappa_j)s^T\Big\{H_i^\dagger - H_j^\dagger\Big\} (\nabla_q P_i)^T P_j\\
	&= s^T\Big\{P_i + P_j - (\kappa_i-\kappa_j)^2H_i^\dagger H_j^\dagger\Big\} (\nabla_q P_i)^T P_j\\
	&= e^TP_j\nabla_q P_iP_i(\nabla_qP_i)^TP_j + 0\\
	&\quad{}+ (\kappa_i-\kappa_j)^2q^T P_i\nabla_e P_jH_i^\dagger H_j^\dagger(\nabla_q P_i)^T P_j\\
	&= \delta_{j,0}\,\kappa_i^2 \nabla v e\cdot \nabla v - q^T P_i(\nabla_e P_j)^T\nabla_q P_i P_j
\end{align*}
since $(\nabla_{P_je} P_j)^T = (\nabla_e P_j)^T(I-P_j)$, $H_i^\dagger$ and $H_j^\dagger$ commutes, and since
\[e^TP_j\nabla_q P_iP_i(\nabla_qP_i)^TP_j = \delta_{j,0}\,\kappa_i^2 e^T\nabla v^Tq^TP_i q \nabla v = \delta_{j,0}\,\kappa_i^2 \nabla ve\cdot \nabla v.\]
Multiplying \eqref{eq:er} with $q^T$ now produces some cancellations and we have shown that
\begin{equation}
0 = q^TP_i\Big\{ (\nabla_q P_i)^T\nabla_eP_j - 2(\nabla_e P_j)^T\nabla_qP_i\Big\}P_j
\label{eq:}
\end{equation}
for all $i\neq j$, $i\geq 1$, and $q$ such that $P_iq = q$, $|q|=1$. The vector $e$ is still arbitrary. Since the first term equals $\kappa_i\kappa_j e^TP_j$ by \eqref{eq:accel_lem} in the Lemma, and since the second term can be written as $q^TP_i(\nabla_e P_j)^T\nabla_qP_iP_j = (\kappa_j-\kappa_i)q^TP_i\nabla_e P_jH_j^\dagger\nabla_qP_iP_j = (\kappa_i-\kappa_j)e^TP_j\nabla_q P_iH_j^\dagger\nabla_qP_iP_j$, it follows that
\begin{equation}
\kappa_i\kappa_j P_j = 2(\kappa_i-\kappa_j)P_j\nabla_q P_iH_j^\dagger\nabla_qP_iP_j.
\label{eq:}
\end{equation}
Divide by $\kappa_i-\kappa_j$, take the trace and sum over $j$ different from $i$.
\begin{align*}
\sum_{\substack{j=0\\j\neq i}}^m \frac{\kappa_i\kappa_j}{\kappa_i-\kappa_j}\tr P_j
	&= 2\sum_{\substack{j=0\\j\neq i}}^m \tr\left(\nabla_q P_iH_j^\dagger\nabla_qP_iP_j\right)\\
	&= 2\sum_{\substack{j=0\\j\neq i}}^m \sum_{\substack{k=0\\k\neq j}}^m \frac{\tr\left(\nabla_q P_iP_k\nabla_qP_iP_j\right)}{\kappa_j-\kappa_k}\\
	&= 2\sum_{\substack{j,k=0\\ j\neq i\neq k\\ k\neq j}}^m \frac{\tr\left(\nabla_q P_iP_k\nabla_qP_iP_j\right)}{\kappa_j-\kappa_k}.
\end{align*}
The last equality is due to $P_i\nabla_qP_iP_j = 0$.
Since the grid of summation indexes is symmetric and since the summand is odd in $j$ and $k$ by the cyclic property of the trace, the sum is zero. Thus
\begin{equation}
\sum_{\substack{j=0\\ j\neq i}}^m d_j\frac{\kappa_i\kappa_j}{\kappa_i-\kappa_j} = 0,\qquad \forall i = 1,\dots,m,
\label{eq:cartan}
\end{equation}
where $d_j := \tr P_j$ is the multiplicity of the eigenvalue $\kappa_j$. This is known as \emph{Cartan's Fundamental Formula}.

Fix $i\geq 1$ so that $\kappa_i$ is the eigenvalue with smallest absolute value. Consider the indexes $j\geq 1$, $j\neq i$. If $\kappa_j<\kappa_i$, we have that $\kappa_i-\kappa_j>0$ and $\kappa_j<0$. Similarly, if $\kappa_j>\kappa_i$, then $\kappa_i-\kappa_j<0$ and $\kappa_j>0$. In either case
\[\frac{\kappa_j}{\kappa_i-\kappa_j} < 0\]
and \eqref{eq:cartan} is a contradiction unless $m=1$.
\end{proof}

\section{Integrating back and assembling the proof}

Let us recall the assumptions of the Theorem:
We have a $C^2$ function $u$ in a connected domain $\Omega\subseteq\mR^n$ and two functions $f$ and $g$ such that
\[
\begin{cases}
\lvert \nabla u\rvert = f(u),\\
\Delta u = g(u).
\end{cases}
\]
Furthermore, $f$ is positive and $C_1 := \Delta_1u(x_0)\geq 0$ at some point $x_0\in\Omega$.

Set $C_0 := u(x_0)$ and define
\[v(x) := F(u(x))\]
where
\[F(t) := \int_{C_0}^t\frac{\dd s}{f(s)},\qquad t\in \cI := u(\Omega),
\]
Note that $v(x_0) = 0$ and that $v\in C^2(\Omega)$ since $f$ is positive and $C^1$ by Proposition \ref{prop:fder}.
Also, $F'>0$ and $F$ has a $C^2$ inverse $U$ satisfying the relations
\begin{equation}
U'(s) = 1/F'(U(s)) = f(U(s)),\qquad U''(s) = f'(U(s))f(U(s)).
\label{eq:rel}
\end{equation}

Moreover,
\[\nabla v = F'(u)\nabla u = \frac{\nabla u}{f(u)}\]
and thus $|\nabla v|= 1$ in $\Omega$. This implies that $\nabla v\cH v = 0$ and the gradient is therefore a unit length eigenvector to the Hessian with corresponding eigenvalue 0. Next,
\begin{equation}
u(x) = U(v(x)),\qquad \nabla u = U'(v)\nabla v,\qquad \cH u = U''(v)\nabla v^T\nabla v + U'(v)\cH v,
\label{eq:next}
\end{equation}
and
\[\Delta v = \frac{g(U(v)) - U''(v)}{U'(v)} =: \tilde{g}(v).\]
In fact, we see that $\Delta_\infty^N u = \nabla v\cH u\nabla v^T = U''(v)$ and $U'(v)= f(u)$, so the Laplacian of $v$ is the 1-Laplacian of $u$:
\begin{equation}
\Delta_1 u := \frac{1}{|\nabla u|}\left(\Delta u - \Delta_\infty^N u\right) = \Delta v.
\label{eq:1laplap}
\end{equation}

By Proposition \ref{prop:eigenvalues are constant on level sets} $v$ is real-analytic in $\Omega$ and has a Hessian on the form
\[\Hv(x) = \sum_{i=0}^m\frac{c_i}{1 + c_iv(x)}P_i(x)\]
in a neighbourhood $\Omega'$ of $x_0$. The constants $c_i$ are distinct and $c_0 = 0$.
If $m=0$ then $\cH v = 0$ and $v$ is affine in $\Omega'$. By analyticity, it must be the same affine function in the whole of $\Omega$. That is,
\[v(x) = q^T(x-x_0)\]
for some unit length constant $q\in\mR^n$ since $v(x_0) = 0$ and $|\nabla v| = 1$. It follows that
\[u(x) = U(v(x)) = U(q^T(x-x_0)),\qquad x\in\Omega.\]
Finally, by \eqref{eq:next} and \eqref{eq:rel}, $g(u) = U''(v) = f'(u)f(u)$ which also implies that $g\in C(\cI)$. 


If $m\geq 1$, then
\[\cH v(x) = \kappa_1(x)P_1(x),\qquad \kappa_1(x) := \frac{c_1}{1 + c_1v(x)},\]
in $\Omega'$ by Proposition \ref{prop:mis1}.
By Lemma \ref{lem:M}, $\nabla_q P_1 P_1 = -\kappa_1(\nabla vq\cdot I + \nabla v^Tq^T)P_1$ and $\nabla_q P_0 P_0 = 0$. But since $P_0 + P_1 = I$ we get that $\nabla_q P_1 = - \nabla_q P_0$ and it follows that
$\nabla_q P_1  = \nabla_q P_1(P_0+P_1) = -\nabla_q P_0 P_0 + \nabla_q P_1P_1 = \nabla_q P_1P_1$. That is,
\[\nabla_q P_1 = -(\nabla vq\cdot I + \nabla v^Tq^T)\cH v.\]
Defining the one-rank projection $G := \nabla v^T\nabla v$, we find that also
\[\nabla_q G = \nabla[\nabla v q\cdot\nabla v^T] = \nabla v q\cdot\cH v + \nabla v^Tq^T\cH v\]
for all $q\in\mR^n$.
Thus, $\nabla[P_1+G] = 0$ and since $P_1G = 0$, 
\[P_1(x) + G(x) =: R_0\]
is a constant projection with rank $k := \tr R_0 = d_1 + 1 \geq 2$.

Consider next the vector field $\nabla v^T/\kappa_1$ in $\Omega'$. Its Jacobian matrix is precisely
\[\nabla\left[\frac{\nabla v^T}{\kappa_1}\right] = \frac{1}{\kappa_1}\cH v - \frac{-\kappa_1^2}{\kappa_1^2}\nabla v^T\nabla v = P_1 + G = R_0,\]
so $\nabla v^T/\kappa_1 = R_0x - x_*$ for some constant $x_*\in\mR^n$. Since $\nabla vR_0 = \nabla v G = \nabla v$, we may write
\[\frac{\nabla v^T(x)}{\kappa_1(x)} = R_0(x-x_*),\]
and taking the length then yields
\begin{equation}
|v(x) + 1/c_1| = \frac{1}{|\kappa_1(x)|} = |R_0(x-x_*)|.
\label{eq:semiexp}
\end{equation}

To find the correct sign, we now use the assumption $C_1 \geq 0$. Since $c_1\neq 0$ and
\begin{equation}
C_1 = \Delta_1 u(x_0) = \Delta v(x_0) = d_1\frac{c_1}{1 + c_1v(x_0)} = (k-1)c_1,
\label{eq:C1}
\end{equation}
it follows that $C_1>0$ and $c_1>0$. Thus, as claimed in the Theorem, $u$ is a function of an affine function if and only if $C_1=0$.

Going back to \eqref{eq:semiexp}, as $v$ is smooth and vanishes at $x=x_0$, the only possibility is
\[F(u(x)) = v(x) = |R_0(x-x_*)| - 1/c_1,\qquad x\in\Omega'\]
which again extends to $\Omega$ by analyticity.
Thus,
\[u(x) = U_k\left(|R_0(x-x_*)|\right),\qquad x\in\Omega,\]
where $U_k$ is the inverse of
\[F_k(t) := \frac{1}{c_1} + F(t) = \frac{k-1}{C_1} + \int^t_{C_0}\frac{\dd s}{f(s)}.\]
Note that $R_0(x-x_*)\neq 0$ for $x\in\Omega$ since $v$ would not be smooth otherwise. In particular, $x_*\notin\Omega$ and $\Omega$ must be a proper subset of $\mR^n$. Also, \eqref{eq:C1} implies that
$v$ is bounded below by $-1/c_1 = -(k-1)/C_1$ which is the negative distance from $x_0\in\Omega$ to the affine subspace $\{R_0(x-x_*) = 0\}$. Thus, $F_k > 0$.

Finally, we note that
\[\Delta v = d_1\frac{c_1}{1 + c_1v} = \frac{(k-1)c_1}{1 + c_1F(u)} = \frac{k-1}{F_k(u)},\]
and by \eqref{eq:rel} and \eqref{eq:next} we get the following relation between $g$ and $f$.
\[g(u) = \Delta u = U''(v) + U'(v)\Delta v = f(u)\left(f'(u) + \frac{k-1}{F_k(u)}\right).\]

This concludes the proof of the Theorem.

\paragraph{Acknowledgments:}
Supported by the Academy of Finland (grant SA13316965), and Aalto University.
I thank Professor Thorbergsson for pointing out the works \cite{MR2999038} and \cite{MR3735906}.

\bibliographystyle{alpha}
\bibliography{C:/Users/Karl_K/Documents/PhD/references}

\end{document}